\documentclass{amsart}

\usepackage{amsmath, amsthm, amssymb}
\usepackage{verbatim}
\usepackage{xspace}
\usepackage{color}

\newtheorem{lemma}{Lemma}
\newtheorem{proposition}{Proposition}
\newtheorem{theorem}{Theorem}

\usepackage{mathtools}
\mathtoolsset{showonlyrefs}

\def \vol{\operatorname {vol}}
\def \w{\operatorname w}
\def \sg{\operatorname{sg}}
\renewcommand{\div}{\operatorname{div}}
\newcommand \affplap[1][p]{\Delta_{#1}^{\mathcal A}}
\newcommand \gnorm[2][p] {\| | \nabla #2 | \|_{#1}}
\def \suchthat {:\ }
\def \Lp {L_p}
\def \esssup {\operatorname{ess sup}}

\def \constE {c_{n,p}} %en \mathcal_pE, la constante que multiplica a la integral en S^{n-1} fuera de la potencia -1/n
\def \constRZ {C_{n,p}(\domain)} %la constante de la nueva desigualdad Reverse-Zhang, que multiplica a las dos normas
\def \constAbsRZ{d_{n,p}} %la constante de la nueva desigualdad Reverse-Zhang, que aparece en la constRZ (no depende del dominio)
\def \constAuxRZ {a_{n,p}} %la constante que aparece en la cuenta de RZ, cuando se compara max_\xi \|\nabla_\xi f\|_p^p > \constAuxRZ \gnorm f

\def \constCentroidPol{r_{n,p}}% en el centroid body, en coordenadas polares, la constante que multiplica dentro de la potencia 1/p, en el denominador
\def \constCentroid{b_{n,p}}% en el centroid body, la constante que multiplica dentro de la potencia 1/p, en el denominador
\def \constchino {m_{n,p}} %la constante del chino que compara \E f con min_sl

\def \constTal {t_{p}} %la constante de la Poincare 1 dimensional de Talenti

\newcommand \vball[1] {\omega_{#1}}

\newcommand \affineeigenvalue[1][\domain] {\lambda^{\mathcal A}_{1,p}(#1)} %affine eigenvalue
\newcommand \affineeigenone[1][\domain] {\lambda^{\mathcal A}_{1,1}(#1)} %affine eigenvalue with p=1
\newcommand \classicaleigenvalue[1][\domain] {\lambda_{1,p}(#1)} %classical eigenvalue
\newcommand \classicaleigenone[1][\domain] {\lambda_{1,1}(#1)} %classical eigenvalue

\def \domain {\Omega}
\newcommand \boundary[1] {\partial #1}
\newcommand \W[1][\domain] {W^{1,p}_0(#1)}
\newcommand \BV[1][\domain] {\operatorname{BV}(#1)}
\newcommand \E[1][p] {\mathcal E_{#1}}
\def \R {\mathbb R}
\def \Rn {\mathbb R^n}
\def \S {{{\mathbb S}^{n-1}}}
\def \B {\mathbb B}
\renewcommand \L[1][f] {L_{p,#1}}
\def \sl{\operatorname{SL_n(\R)}}
\def \gl{\operatorname{GL_n(\R)}}

\def \PZ {Polya-Szeg\"o\xspace}
\def \FK {Faber-Krahn\xspace}
\def \BZ {Brothers-Ziemer\xspace}
\def \SZ {Sobolev-Zhang\xspace}
\def \RK {Rellich-Kondrachov\xspace}
\def \PP {Petty projection\xspace}
\def \BP {Busemann-Petty\xspace}
\def \Holder {H\"older\xspace}
\def \Poincare {Poincar\'e\xspace}
\def \EL {Euler-Lagrange\xspace}
\def \BS {Blaschke-Santal\'o\xspace}

\author{J. Haddad}
\address{Juli\'an Haddad: Departamento de Matem\'atica, ICEx,  Universidade Federal de Minas Gerais, 30.12370, Belo Horizonte, Brasil.}
\email{jhaddad@mat.ufmg.br}

\author{C. H. Jim\'enez}
\address{Carlos Hugo Jim\'enez:Departamento de Matem\'atica, Pontif\'icia Universidade Cat\'olica do Rio de Janeiro}
\email{hugojimenez@mat.puc-rio.br}

\author{M. Montenegro}
\address{Marcos Montenegro: Departamento de Matem\'atica, ICEx,  Universidade Federal de Minas Gerais, 30.123970, Belo Horizonte, Brasil.}
\email{montene@mat.ufmg.br}

\begin{document}
\title{From affine Poincar\'e inequalities to affine spectral inequalities}
\subjclass{Primary 35P30; Secondary 52A40, 35B09}
\date{}
\begin{abstract}
Given a bounded open subset $\domain$ of $\R^n$, we establish the weak closure of the affine ball $B^{\mathcal A}_p(\domain) = \{f \in \W:\ \E f \leq 1\}$ with respect to the affine functional $\E f$ introduced by Lutwak, Yang and Zhang in \cite{lutwak2002sharp} as well as its compactness in $L^p(\domain)$ for any $p \geq 1$. These points use strongly the celebrated Blaschke-Santal\'{o} inequality. As counterpart, we develop the basic theory of $p$-Rayleigh quotients in bounded domains, in the affine case, for $p\geq 1$. More specifically, we establish $p$-affine versions of the \Poincare inequality and some of their consequences. We introduce the affine invariant $p$-Laplace operator $\affplap f$ defining the \EL equation of the minimization problem of the $p$-affine Rayleigh quotient. We also study its first eigenvalue $\affineeigenvalue$ which satisfies the corresponding affine \FK inequality, this is that $\affineeigenvalue$ is minimized (among sets of equal volume) only when $\domain$ is an ellipsoid. This point depends fundamentally on PDEs regularity analysis aimed at the operator $\affplap f$. We also present some comparisons between affine and classical eigenvalues, including a result of rigidity through the characterization of equality cases for $p \geq 1$. All affine inequalities obtained are stronger and directly imply the classical ones.
	%Lastly, for $p=1$ and $\domain$ convex we find a sufficient condition of spectral type for the domain $\domain$ to be in John's position. All affine inequalities obtained are stronger and directly imply the classical ones.	

%We prove compactness properties of the affine functional $\E f$ defined by Lutwak, Yang, Zhang \cite{lutwak2002sharp} and as consequence we develop the basic theory of $p$-Rayleight quotients in bounded domains, in the affine case, for $p\geq 1$. More specifically we obtain $p$-affine versions of the \Poincare inequality and the \RK immersion theorem. We define an affine invariant $p$-Laplace operator $\affplap f$ defining the \EL equation of the minimization problem of the $p$-affine Rayleight quotient. We also study its first eigenvalue $\affineeigenvalue$ which satisfies the corresponding affine \FK inequality, this is that $\affineeigenvalue$ is minimized (among sets of equal volume) when $\domain$ is an ellipsoid. All the affine inequalities obtained are stronger and directly imply the classical ones.
%For $p=1$ and $\domain$ convex we find a sufficient condition of spectral type, for the domain $\domain$ to be in John's position.
\keywords{\Poincare inequalities, \FK inequalities, affine-invariant, affine $p$-Laplacian}
\end{abstract}
\maketitle
\sloppy

\section{Introduction}
Sharp functional inequalities are among the fundamental tools in the developing of mathematics with applications in various branches of science. One of them is the classical $L^p$ \Poincare inequality on bounded open sets $\domain \subset \R^n$, which has its origin in the seminal work of \Poincare \cite{poincare}, and states, for any $p \geq 1$ and $f \in C_0^\infty(\domain)$, that

\begin{equation}
	\label{ineq_classicalpoincare}
D_0(\domain) \int_\domain |f|^p dx \leq \int_\domain |\nabla f|^p dx,
\end{equation}
where $D_0(\domain)$ is a positive constant depending only on the open set $\domain$. The inequality \eqref{ineq_classicalpoincare} naturally extends to the completion $\W$ of the space $C_0^\infty(\domain)$ of smooth functions compactly supported in $\domain$, with respect to the norm
\[
\|f\|_{\W} := \left( \int_\domain |\nabla f|^p dx \right)^{1/p}.
\]

Consider the optimal constant related to \eqref{ineq_classicalpoincare} given by

\begin{equation}
	\label{def_classicaleigenvalue}
\classicaleigenvalue := \inf_{f \in \W \setminus \{0\}} R_p(f),
\end{equation}
where $R_p(f)$ denotes the Rayleigh $p$-quotient:

\[
R_p(f) := \frac{\int_\domain |\nabla f|^p dx}{\int_\domain |f|^p dx}.
\]
For $p > 1$, the fact that $\W$ is a reflexive Banach space with respect to the norm $\gnorm f$ and the compactness of the embedding $\W \hookrightarrow L^p(\domain)$ ensure that the infimum $\classicaleigenvalue$ is attained by a function $f_p \in \W$. Moreover, $\classicaleigenvalue$ is the first eigenvalue of the $p$-Laplace (or $p$-Laplacian) operator
\[
\Delta_p f := - {\div}(|\nabla f|^{p-2} \nabla f)
\]
and, consequently, $f_p$ is a bounded first eigenfunction, which can be assumed positive, in $C^{1,\alpha}(\domain)$ if $\domain$ is non-smooth and in $C^{1,\alpha}(\overline{\domain})$ if $\domain$ is smooth. We refer for example to \cite{de1960sulla}, \cite{giaquinta1984quasi} and \cite{di1984harnack} for boundedness and local $C^\alpha$ regularity within the quasi-minima theory in calculus of variations, to \cite{dibenedetto1982c}, \cite{ladyzhenskaya1968linear}, \cite{tolksdorf1984regularity} and \cite{uhlenbeck1977regularity} for the boundedness and local $C^\alpha$ and $C^{1,\alpha}$ regularity and to \cite{tolksdorf1984regularity} and \cite{lieberman1988boundary} for global $C^{1,\alpha}$ regularity within the theory of quasilinear elliptic equations in divergence form.     In particular, in latter, one deduces that $f_p$ is positive in $\domain$ (e.g. \cite{pucci2004strong,pucci2007j}) and unique, up to a multiplicative constant (e.g. \cite{lindqvist1990equation}). For $p = 1$, $\lambda_{1,1}(\domain)$ is not attained by any function in $W^{1,1}_0(\domain)$, but by a function $f_1 \in BV(\domain)$, see \cite{kawohl2007dirichlet} and references therein.

Some fundamental questions in mathematics, part of them originated in physics, are formulated in terms of bounds of eigenvalues associated to certain differential operators (particularly the Laplace operator) in an area known as spectral geometry. For a complete overview on problems of great interest on this subject, we refer to the classical monographs \cite{bandle1980isoperimetric}, \cite{osserman1978isoperimetric}, \cite{payne1967isoperimetric} and to the excellent recent surveys \cite{ashbaugh2007isoperimetric} and \cite{henrot2006extremum}.

One of the most famous questions in spectral geometry was posed in 1887 by Rayleigh in the book \cite{strutt1945theory} entitled {\it The theory of sound}. In occasion, he conjectured that among all membranes (open sets $\domain \subset \R^2$) of same area, the disk minimizes the corresponding principal frequencies of sounds (eigenvalues $\lambda_{1,2}(\domain)$). The conjecture was proved in the 1920s, independently, by Faber \cite{faber1923beweis} and Krahn \cite{krahn1925rayleigh} for arbitrary dimensions, namely they established for any $n \geq 2$ the celebrated \FK isoperimetric inequality which states that

\begin{equation}
	\label{ineq_classicalFKp1}
\lambda_{1,2}(\domain) \geq \lambda_{1,2}(\B)
\end{equation}
for every open set $\domain$ in $\R^n$ having the same measure of a fixed ball $\B$. Moreover, equality holds if, and only if, $\domain$ is a ball.

Once the value of $\lambda_{1,2}(\B)$ is explicitly known, inequality \eqref{ineq_classicalFKp1} can be rephrased as (e.g. \cite{krahn1926minimaleigenschaften})

\begin{equation}
|\domain|^{2/n} \lambda_{1,2}(\domain) \geq \frac{\pi j^2_{(n-2)/2}}{\Gamma^{2/n}((n+2)/2)},
\end{equation}
where $j_{(n-2)/2,1}$ denotes the first positive zero of the Bessel function $J_{(n-2)/2}(x)$. The original proof of the \FK inequality makes use of Schwarz symmetrization (spherically symmetric decreasing rearrangement). For different proofs of the \FK inequality we mention \cite{bucur2012new} and \cite{kesavan1988some}.

Other \FK inequalities associated to more general elliptic operators have also been considered, particularly to the $p$-Laplace operator. In this case, it has been proved by Alvino, Ferone and Trombetti \cite{alvino1998properties} for $p > 1$ (see also \cite{bhattacharya1999proof} and \cite{matei2000first}) and by Fusco, Maggi and Pratelli \cite{fusco2009stability} for $p = 1$ that the $p$-\FK inequality
\begin{equation} \label{ineq_classicalFK}
\classicaleigenvalue \geq \classicaleigenvalue[\B]
\end{equation}
holds for every open set $\domain$ in $\R^n$ having the same measure as a fixed ball $\B$. Moreover, equality holds if, and only if, $\domain$ is a ball. The major difficulty is ensuring that the equality occurs only on balls. This one is surrounded for $p > 1$ thanks to the celebrated theorem of Brothers and Ziemer of \cite{Brothers1988} (see page 154) on the characterization of equality in the \PZ principle, and for $p = 1$ is used that $\classicaleigenvalue$ converges to the Cheeger constant $h_1(\domain)$ as $p \rightarrow 1^+$ as shown by Kawohl and Fridman \cite{kawohl2003isoperimetric} and that a suitable quantitative form of the Cheeger isoperimetric inequality occurs (see \cite{fusco2009stability} on page 57).

Two important landmarks in the modern theory of sharp functional inequalities are the bedrock works due to Zhang \cite{zhang1999affine} and Lutwak, Yang and Zhang \cite{lutwak2002sharp} connecting areas as Analysis and Convex Geometry. Indeed, for $p \geq 1$, let
\begin{equation}
	\label{def_affineterm}
	\E f = \constE \left(\int_{\S} \|\nabla_\xi f(x)\|_p^{-n}d\xi\right)^{-\frac 1n}
\end{equation}
with
\[
\constE = \left(n \vball n\right)^{\frac 1n} \left(\frac{n \vball n \vball {p-1}}{2 \vball {n+p-2}}\right)^{\frac 1p},
\]
where $\nabla_\xi f(x) = \nabla f(x) \cdot \xi$ and $\vball k$ denotes the volume of the unit ball in $\R^k$. Denote by $W^{1,p}(\R^n)$ the space of weakly differentiable functions in $\R^n$ endowed with the $L^p$ gradient norm.

Let $1 \leq p < n$. The sharp affine $L^p$ Sobolev inequality, proved in \cite{lutwak2002sharp} for $1 < p < n$ and in \cite{zhang1999affine} for $p = 1$, states that for any $f \in W^{1,p}(\R^n)$
\begin{equation}
	\label{ineq_Zhang}
\|f\|_{\frac{n p}{n-p} } \leq S_{n,p} \E f.
\end{equation}
Moreover, equality holds for $p>1$ if, and only if,
\[
f(x) = a\left(1 + b|A(x - x_0)|^\frac p{p-1}\right)^{1-\frac pn}
\]
for some $a \in \R$, $b > 0$, $x_0 \in \R^n$ and $A \in \gl$, where $\gl$ denotes the set of invertible $n \times n$-matrices.
For $p=1$, equality is attained for multiples of characteristic functions of ellipsoids, which belong to the larger space $\BV$, the space of functions of bounded variation.

For more references on optimal affine functional inequalities we quote \cite{haddad2016sharp, cianchi2009affine, haberl2009asymmetric, ludwig2011sharp, lutwak2006optimal, wang2012affine, de2018sharp, haddad2019sharp}.
We will refer to inequality \eqref{ineq_Zhang} as the \SZ inequality.

When one restricts \eqref{ineq_Zhang} to functions in the set $\W := \{f \in H^{1,p}(\R^n):\, f = 0 {\rm\ on\ }\, \R^n \setminus \domain\}$ and makes use of \Holder's inequality, one easily obtains the following inequality for any $f \in \W$:

\begin{equation}
	\label{ineq_affinepoincare}
K(p,\domain) \int_\domain |f|^p dx \leq \E^pf,
\end{equation}
where $K(p,\domain)$ is a positive constant depending only on the parameter $p$ and the bounded open set $\domain$. In other words, as a consequence of the sharp affine Sobolev inequality \eqref{ineq_Zhang}, we deduce that the affine \Poincare inequality holds on $\W$ for any $1 \leq p < n$.

A first question then arises:

\begin{center}
Let $f \in L^p(\domain)$ be a weakly differentiable function such that $f = 0$ on $\R^n \setminus \domain$.\\ Does $\E f < \infty$ imply $\gnorm f < \infty$?
\end{center}
As is well known, the reciprocal is true since the inequality $\E f \leq \gnorm{f}$ always holds for any $p \geq 1$, see for example \cite{lutwak2002sharp}. The above query will be affirmatively answered in Section 2 for general bounded open sets by means of a type of reverse inequality, so that $\E f < \infty$ if, and only if, $\gnorm f < \infty$ for any $p \geq 1$. Consequently, $\W$ is the adequate space for dealing with affine \Poincare inequalities, and so we introduce for each $p \geq 1$:

\begin{equation}
	\label{def_affineeigenvalue}
\affineeigenvalue := \inf_{f \in \W \setminus \{0\}} R^{{\mathcal A}}_p(f),
\end{equation}
where $R^{{\mathcal A}}_p(f)$ denotes the affine Rayleigh $p$-quotient:

\[
R^{{\mathcal A}}_p(f) := \frac{\E^pf}{\|f\|^p_{L^p(\domain)}}.
\]
It deserves to be noticed that $R^{{\mathcal A}}_p(f)$ and $\affineeigenvalue$ are affine invariants with respect to volume preserving affine transformations.
As counterparts of the above definition, a number of interface questions emerge connecting Analysis, Convex Geometry and Spectral Geometry.

The present paper focuses on the following issues:

\begin{itemize}

\item[(A)] Is the number $\affineeigenvalue$ positive for any $p \geq 1$? \\ Notice that the positivity of $\affineeigenvalue$ means that the affine $L^p$ \Poincare inequality \eqref{ineq_affinepoincare} holds on bounded open sets, as it is known only for $p \leq n$.

\item[(B)] Is the inclusion $B^{\mathcal A}_p(\domain):= \{f \in \W:\ \E f \leq 1\} \hookrightarrow L^p(\domain)$ compact for any $p \geq 1$?\\ The compactness is not clear since the set $B^{\mathcal A}_p(\domain)$ is not bounded in $\W$. A counter-example of the latter is also provided.

\item[(C)] Is the infimum $\affineeigenvalue$ attained for some function $f^{{\mathcal A}}_p \in \W$ for any $p \geq 1$?\\ The attainability of $\affineeigenvalue$ is by far not direct since the Zhang's term $\E f$ is not a convex functional.

\item[(D)] Is there any analytical bridge connecting $\affineeigenvalue$ to the spectrum of some differential operator in case $p > 1$?

\item[(E)] Are all minimizers of $\affineeigenvalue$ smooth for $p > 1$? Is there any characterization of them?\\ The questions (D) and (E) link the best constants of affine \Poincare inequalities to PDEs theory.

\item[(F)] Does the affine \FK inequality hold for any $p \geq 1$?

\item[(G)] If so, is it possible to characterize all cases of equality?\\ The questions (F) and (G) link Geometry Convex and Spectral Geometry. The first one consists in finding an optimal bounded open subset $\mathbb E$ of $\R^n$ in the sense that the affine $p$-\FK inequality
\begin{equation}
	\label{ineq_affineFK}
	\affineeigenvalue \geq \affineeigenvalue[\mathbb E]
\end{equation}
holds for every bounded open subset $\domain \subset \R^n$ with the same Lebesgue measure as $\mathbb E$ and the second one in characterizing all cases of equality in \eqref{ineq_affineFK}.

\item[(H)] How far apart are the best constants $\affineeigenvalue$ and $\classicaleigenvalue$?\\ The idea is to compare these two numbers in terms of geometric properties of the domain $\domain$ and to establish properties of rigidity.

\end{itemize}

The central goal here is to provide answers to the raised questions from (A) to (H).

The first three theorems give complete answers to (A), (B) and (C).

\begin{theorem}
	\label{thm_affinePI}
Let $\domain \subset \R^n$ be a bounded open set and $p \geq 1$. Then the affine $L^p$ \Poincare inequality \eqref{ineq_affinepoincare} holds for every function $f \in \W$.
\end{theorem}

\begin{theorem}
	\label{thm_affineRK}
Let $\domain \subset \R^n$ be a bounded open set and $p \geq 1$. Then the set $B^{\mathcal A}_p(\domain)$ is compactly immersed into $L^p(\domain)$ and is unbounded in $\W$.
\end{theorem}

\begin{theorem}
	\label{thm_existence}
Let $\domain \subset \R^n$ be a bounded open set and $p \geq 1$.
	Then the infimum $\affineeigenvalue$ in \eqref{def_affineeigenvalue} is attained for some function $f^{{\mathcal A}}_p \in \W$, if $p>1$, and $f^{{\mathcal A}}_1 \in \BV$, if $p=1$.
\end{theorem}

Before we go further and state more results, we present a little bit of definition.

For $p > 1$, we introduce the affine $p$-Laplace operator $\affplap$ on $\W$ as the non-local quasilinear operator in divergence form given by

\begin{equation}
	\label{def_affinelaplacian}
\affplap f := -{\div} \left( H_{f}^{p-1}(\nabla f) \nabla H_{f}(\nabla f) \right),
\end{equation}
where

\begin{equation} \label{def_Hf}
	H_{f}^p(v):= \constE^{-n} \E^{p+n}(f) \int_{\S} \| \nabla_\xi f \|_p^{-n-p} |\langle \xi, v \rangle|^p\, d\xi.
\end{equation}
When $p = 2$, the operator $\affplap[2]$ coincides with the affine laplacian introduced by Schindler and Tintarev in \cite{schindler2018compactness} by means of an interesting property of $\mathcal E_2$ that works in the specific case $p=2$.

As we shall see, the operator $\affplap$ satisfies two fundamental properties that justify its name. Firstly, the operators $\affplap$ and $\Delta_p$ coincide for radial functions
and, secondly, $\affplap$ verifies the affine invariance property: for any $f \in \W$ and $T \in \sl$,

\[
\affplap (f \circ T) = (\affplap f) \circ T \ \ {\rm on}\ \ T^{-1}(\domain),
\]
where $\sl$ denotes the special linear group of $n \times n$ matrices with determinant equal to $1$.

The affine $p$-Laplace operator for $p > 1$ will appear in connection with the derivative of the Zhang term $\E f$ with respect to $f$. In particular, it will be shown that a minimizer $f^{\mathcal A}_p \in \W \setminus \{0\}$ for $\affineeigenvalue$ is a weak solution of the $(p-1)$-homogeneous equation

\begin{equation}
		\label{eq_ELZhang}
		 \affplap f = \affineeigenvalue |f|^{p-2} f \ \ {\rm in}\ \ \domain.
\end{equation}
Notice that \eqref{eq_ELZhang} along with the definition of $\affineeigenvalue$ implies that this is the smallest (not necessarily simple) eigenvalue of $\affplap$ on $\W$ with associated eigenfunction $f^{\mathcal A}_p$.

The issue (D) and part of (E) are addressed in the following result:
	
\begin{theorem} \label{thm_ELZhang}
	Let $\domain \subset \R^n$ be a bounded open set and $p > 1$. Then, $f^{{\mathcal A}}_p \in \W$ is an eigenfunction of the operator $\affplap$ on $\W$ corresponding to $\affineeigenvalue$ in the sense of \eqref{eq_ELZhang} if, and only if, it minimizes the affine Rayleigh $p$-quotient $R^{{\mathcal A}}_p(f)$. In particular, $\affineeigenvalue$ is the smallest among all real Dirichlet eigenvalues of $\affplap$. Moreover, each eigenfunction (or minimizer) corresponding to $\affineeigenvalue$ is a bounded function in $C^{1, \alpha}(\domain)$ (and in $C^{1, \alpha}(\overline{\domain})$ whenever $\domain$ has $C^{2, \alpha}$ boundary) which does not change sign when $\domain$ is connected.
\end{theorem}
As an immediate consequence of this result, each eigenfunction (or minimizer) $f^{{\mathcal A}}_p \in \W$ of $\affineeigenvalue$ has defined sign and the set of eigenfunctions (or minimizers) associated to $\affineeigenvalue$ contains at least the two half straights $\{t f^{{\mathcal A}}_p:\, t > 0\}$ and $\{t f^{{\mathcal A}}_p:\, t < 0\}$.

The questions (F) and (G) are answered in the following result:

\begin{theorem} \label{thm_affineFK}
Let $p \geq 1$ and $\mathbb E$ be an Euclidean ellipsoid in $\R^n$. Then, the affine $p$-\FK inequality \eqref{ineq_affineFK} holds for every bounded open subset $\domain \subset \R^n$ with the same Lebesgue measure as $\mathbb E$. Moreover, equality holds in \eqref{ineq_affineFK} if, and only if, $\domain$ is an Euclidean ellipsoid of the same measure.
\end{theorem}

The next result compares the first affine and classical eigenvalues $\affineeigenvalue$ and $\lambda_{1,p}(\domain)$ for any $p \geq 1$.

\begin{theorem}
	\label{thm_Lambdaproperties}
	Let $\domain \subset \R^n$ be a bounded open set and $p \geq 1$. Then the following inequalities hold:
	\begin{enumerate}

		\item[(a)] $\affineeigenvalue \leq \classicaleigenvalue$,

		\item[(b)] $\affineeigenvalue \geq \constRZ^p \classicaleigenvalue^{1/n}$,

		\item[(c)] $\affineeigenvalue \geq \constchino \min_{T \in \sl} \classicaleigenvalue[T(\domain)]$,

	\end{enumerate}
where
\[
\constRZ:= \constAbsRZ\left( \max_{\xi \in \S} w(\domain, \xi) \right)^{-\frac {n-1}{n}}.
\]
with $w(\domain, \xi)$ denoting the width of $\domain$ in the direction $\xi$ and
$\constAbsRZ$ and $\constchino$ being as in Theorems \ref{thm_reversezhang} and \ref{thm_chino}, respectively.
%\[
%\constchino := \frac{\pi^{\frac 1{2} + \frac p2} \Gamma(\frac{n+p}2) \Gamma( 1 + \frac np)^{\frac pn}}{2^{p + 1} \Gamma(1+\frac n2)^{\frac pn + 1} \Gamma(\frac{p+1}2) \Gamma(1+\frac 1p)^p}.
%\]
	Moreover, inequality in {\normalfont (b)} is always strict.
\end{theorem}
An immediate implication from the assertions (a) and (b) is the lower estimate in the unit ball $\B$:
\[
	\classicaleigenvalue[\B] > C_{n,p}(\B)^\frac{np}{n-1} =
		n 2^{\frac{p}{n-1}-1} (p-1)^{1-p} p^{-p}
		\left(\frac{\pi -\frac{\pi }{p}}{\sin \left(\frac{\pi }{p}\right)}\right)^p
		\frac{\vball {p-1} \vball {n-1}^{\frac{p}{n-1}} \vball n^{1-\frac{p}{n-1}}}{\vball {n+p-2}}
\]

The term appearing at the right-hand side of the part (c) of Theorem \ref{thm_Lambdaproperties} can be regarded as a modification of $\classicaleigenvalue$ that makes it an affine invariant functional which is also bounded when restricted to convex bodies. This modification principle has been investigated for many functionals such as the surface area measure, the in-radius for which one can obtain reverse inequalities in the family of convex bodies (see \cite[Chapter 10.13]{schneider2014convex}), among others. In \cite{bucur2016blaschke, bucur2018reverse}, Bucur and Fragal\'a characterized the maximizers of $\min_{T\in \sl} \lambda_{1,2}(T(\Omega))$ and $\min_{T\in \sl} h_1(\Omega)$ for $\Omega \subseteq \R^2$ convex, obtaining a reverse \FK inequality, with triangles as extremal cases.

By using the monotonicity of $\affineeigenvalue$ with respect to inclusions and John's Ellipsoid Theorem, one can easily prove that $\affineeigenvalue$ is also bounded above in the family of convex bodies with a fixed volume. The problem of finding the upper bound of \eqref{ineq_affineFK} is open and appears to be challenging (see Section 7).

	The statement (a) of Theorem \ref{thm_Lambdaproperties} leads us to the following rigidity theorem:
	\begin{theorem}
		\label{thm_equalitycase}
		Let $\domain \subset \R^n$ be a bounded open set. Then, it holds that:
		%for $p\geq 1$ then $\affineeigenvalue = \classicaleigenvalue$ if, and only if, $\domain$ is a ball.
		\begin{enumerate}
			\item[(a)] If $p > 1$ then $\affineeigenvalue = \classicaleigenvalue$ if, and only if, $\domain$ is a ball.
			%\item[(b)] If $\affineeigenone = \classicaleigenone$ then the minimizer of both eigenvalues can be taken to be the characteristic function of a ball, that is also the maximal volume ellipsoid inside $\domain$. Moreover, if $\affineeigenone = \classicaleigenone = n$ and $\domain$ is convex then, after a translation, $\domain$ is in John's position.
			\item[(b)] If $\affineeigenone = \classicaleigenone$ then the minimizer of both eigenvalues can be taken to be the characteristic function of a ball whose boundary is contained in the boundary of $\domain$.
				In particular, if $\domain$ is convex then it is an euclidean ball.
		\end{enumerate}
	\end{theorem}

Finally, it would be very important to provide some properties of the minimizers of $\affineeigenone$. To this end, we briefly recall the definition of Cheeger sets.

	For $\domain \subset \R^n$ a bounded open set, its Cheeger constant is defined by the infimum
	\begin{equation}
		\label{def_affineCheegerConstant}
		h_1(\domain) = \inf_{C \subseteq \domain} \frac{S(C)}{\vol(C)}
	\end{equation}
	where $S(C)$ is the surface area measure of $C$, and $C$ runs over all sets of finite perimeter (this is, such that $\chi_C \in BV(\domain)$).
	The minimum is attained by a set $C_0$ and it is known that $\chi_{C_0}$ minimizes $\classicaleigenone$.
	This minimizer set $C_0$ is called a Cheeger set of $\domain$ and $h_1(\domain) = \classicaleigenone$.

	We provide a similar relationship in the affine case, with a small twist concerning the size of the affine Cheeger set.
	For any compact set $C \subset \R^n$, a position of $C$ is a set of the form $A C + x_0$ where $A \in \gl$ and $x_0 \in \R^n$.
	We say $C \subset \domain$ is in position of maximal volume inside $\domain$ if $\vol(C) \geq \vol(A C + x_0)$ for any such position $A C + x_0 \subseteq \overline \domain$.
		
	Now let us recall the following celebrated result due to F. John. In his seminal work \cite{JohnF} (see also \cite{Ball1}), John characterized the ellipsoid of maximal volume inside a convex body in terms of its contact points. This result, originally obtained by an optimization argument, has become one of the main tools in the study of Banach-Mazur distance between convex bodies. John's characterization is perhaps better understood when we look at the affine map $T$ such that for a convex body $K\subset \R^n$ we have that the Euclidean unit ball is the ellipsoid of maximal volume inside $TK$. The convex body $TK$, under these circumstances, is said to be in John's position.
	
	Thus, if a convex body $K$ is in {\it John's position} (i.e. the Euclidean unit ball $\B \subseteq K$ is the ellipsoid of maximal volume inside $K$) there exist vectors $u_1, ..., u_m\in \boundary{K}\cap \boundary{\B}$ and positive numbers $c_1,...,c_m$ for some $m\in\mathbb{N}$ such that
	\[\sum_{i=1}^m c_iu_i=0\quad \quad \mbox{ and} \quad \quad \sum_{i=1}^m c_iu_i\otimes u_i=I_n,\]
	where $\boundary{K}$ denotes the boundary of the set $K$, $u_i\otimes u_i$ is the rank one projection in the direction $u_i$ and $I_n$ denotes the identity operator in $\R^n$.
	
	There are different extensions of John's theorem, these extensions often substitute ellipsoids by arbitrary convex bodies. In this direction, we have the characterizations obtained under slightly different assumptions by Giannopoulos, Perissinaki and Tsolomitis \cite{GianPerTso}, Bastero and Romance \cite{BasteroRom} and finally by Gordon, Meyer, Litvak and Pajor in the general case \cite{GLMP}.
	For example, in \cite{GLMP} it is stated that if $\domain$ is a convex body and $C \subseteq \domain$ is a compact set in position of maximal volume inside $\domain$, then for any $z \in C$, we can find contact points $v_1, \ldots, v_m$ of $\domain-z$ and $C-z$, contact points $u_1, \ldots, u_m$ of the polar bodies $(\domain-z)^\circ$ and $(C-z)^\circ$, and positive real numbers $c_1, \ldots, c_m$, such that
	\[\sum_{i=1}^m c_i u_i = 0, \quad \quad \langle u_i, v_i \rangle = 1 \quad \quad \mbox{ and } \quad \quad \sum_{i=1}^m c_i u_i\otimes v_i = I_n.\]
	
 We have as well the work by Gruber and Schuster \cite{GruSchuster} where they provided a beautiful proof of John's theorem using an idea of Voronoi to represent ellipsoids by points on a space of much larger dimension. Finally, there is a functional extension by Alonso-Guti\'errez, Gonz\'alez, Jim\'enez and Villa \cite{AGJV1} where the ellipsoids are replaced by a special kind of ``ellipsoidal'' functions and the convex bodies by log-concave functions.

	Let us go back to the affine Cheeger sets. We verify the following
	\begin{theorem}
		\label{thm_affineCheegerSet}
		Let $\domain \subset \R^n$ be a bounded open set. Then the quantity, here called affine Cheeger constant,
		\[h_1^{\mathcal A}(C) := \frac{\E[1] \chi_C}{||\chi_C||_1} = \frac{ (c_{n,1} \vol(\Pi^\circ C))^{-1/n}}{\vol(C)}\]
		is minimized among all measurable sets of finite perimeter by a set $C_0$, referred as affine Cheeger set, for which $\chi_{C_0}$ is a minimizer of $\affineeigenone$. Moreover, $C_0$ is in position of maximal volume inside $\domain$.
	\end{theorem}
	From this point forward, many questions remain open regarding the affine Cheeger sets.
	Most importantly, whether there is a connection to John's position.
	Considering the extensive work around John's theorem and all the extensions mentioned above, we find striking that, to our knowledge, there are no analytical (or more precisely, spectral) conditions available in the literature for a convex body to be in John's position.
	We mention these and some other questions in the last section.

We remark that there is no quantitative form of the affine \PZ principle proved in the literature, so the existence of a minimizer of $\affineeigenvalue$ (provided by Theorem \ref{thm_existence}) is critical for our proof of Theorems \ref{thm_affineFK}, \ref{thm_Lambdaproperties}, \ref{thm_equalitycase} and \ref{thm_affineCheegerSet}.

This work adds $\affineeigenvalue$, to the already long list of affine-invariant functionals defined on convex or non-convex sets, that get minimized or maximized precisely in the family of ellipsoids.
It provides an affine invariant version of a classical functional $\classicaleigenvalue$, and the affine inequality turns out to be stronger than the classical one (see Theorem \ref{thm_Lambdaproperties}).

The plan of the paper is as follows:

\begin{itemize}
	\item[Section 2:] Preliminaries on convex geometry - We recall some definitions and notations within the theory of convex sets as well as some closely related inequalities. We also highlight the affine \PZ principle and the affine Brothers-Ziemer result proved by Nguyen in \cite{nguyen2015new}.\\

\item[Section 3:] The affine \Poincare inequality\\ Subsection 3.1: A reverse comparison inequality for bounded open sets - We establish a reverse version of the classical inequality $\E f \leq \gnorm{f}$ which will be fundamental in the proof of Theorem \ref{thm_affinePI}. Its proof bases on the famous \BS inequality of the theory of convex geometry.\\ Subsection 3.2: Proof of Theorems \ref{thm_affinePI}, \ref{thm_affineRK} and \ref{thm_existence} - The previous subsection inequality plays an essential role in the proof of these results.\\

\item[Section 4:] Properties of extremal functions\\
Subsection 4.1: The affine \EL equation and associated operator - From explicit computation of the derivative of the Rayleigh functional $f \in \W \mapsto R^{{\mathcal A}}_p(f)$ for $p > 1$, we discover the new operator $\affplap$ underlying the Zhang's term $\E f$. As consequence, we present an important bridge between the best \Poincare constants and Dirichlet principal eigenvalues of PDEs in the affine context. This subsection proves part of Theorem \ref{thm_ELZhang}.
		\\
		Subsection 4.2: Affine invariance properties - The affine $p$-Laplace operator and some of its main properties - We show that the operator is invariant under affine transformations and becomes the $p$-Laplace operator on balls when one restricts to radial functions.\\
Subsection 4.3: Regularity properties and proof of Theorem \ref{thm_ELZhang} - We show that $\affplap$ is an affine non-local elliptic degenerate operator with $C^{1,\alpha}$ regularization property. As a by-product, we prove the remaining statements of Theorem \ref{thm_ELZhang} concerning the smoothness of eigenfunctions.\\

\item[Section 5:] The affine \FK inequality - We provide the complete proof of Theorem \ref{thm_affineFK} for any $p \geq 1$. Here the key ingredients are the sharp forms of $L^1$ classical Sobolev and affine Sobolev-Zhang inequalities and the affine \PZ principle.\\

\item[Section 6:] Comparison of eigenvalues - Again using the inequality $\E f \leq \gnorm{f}$ and its reverse counterpart of the subsection 3.1, we prove the comparisons between $\affineeigenvalue$ and $\classicaleigenvalue$ stated in Theorem \ref{thm_Lambdaproperties} and the statement of rigidity stated in Theorem \ref{thm_equalitycase}. We also prove Theorem \ref{thm_affineCheegerSet} that shows the existence of affine Cheeger sets.\\

\item[Section 7:] Open problems - We raise some questions of interest closely related to affine \Poincare inequalities and their analytical and geometric connections.
\end{itemize}

\section{Preliminaries on convex geometry}
\label{sec_preliminaries}
This section is devoted to basic definitions and notations within the convex geometry. For a comprehensive reference in convex geometry we refer to the book \cite{schneider2014convex}.

We recall that a convex body $K\subset\Rn$ is a convex compact subset of $\Rn$ with non-empty interior.
The support function $h_K$ (denoted also by $h(K,\cdot)$) is defined as
\[
h_K(y) = h(K,y) = \max\{\langle y, z\rangle \suchthat z\in K\}\, .
\]
It describes the (signed) distance of supporting hyperplanes of $K$ to the origin and uniquely characterizes $K$. If $K$ contains the origin in the interior, then we also have the gauge $\|\cdot\|_K$ and radial $r_K(\cdot)$ functions of $K$ defined respectively as
\[
\|y\|_K:=\inf\{\lambda>0 \suchthat  y\in \lambda K\}\, ,\quad y\in\Rn\setminus\{0\}\, ,
\]
\[
r_K(y):=\max\{\lambda>0 \suchthat \lambda y\in K\}\, ,\quad y\in\Rn\setminus\{0\}\, .
\]
Clearly, $\|y\|_K=\frac{1}{r_K(y)}$. We also recall that $\|\cdot\|_K$ is actually a norm when the convex body $K$ is centrally symmetric, i.e. $K=-K$, and the unit ball with respect to $\|\cdot\|_K$ is just $K$.
On the other hand, a general norm on $\Rn$ is uniquely determined by its unit ball, which is a centrally symmetric convex body.

For a convex body $K\subset \Rn$ containing the origin in its interior we define the polar body, denoted by $K^\circ$, by
\[
K^\circ:=\{y\in\Rn \suchthat \langle y,z \rangle\leq 1\quad \forall z\in K\}\, .
\]
Evidently, $h_K^{-1} = r_{K^{\circ}}$. It is also easy to see that $(\lambda K)^\circ=\frac{1}{\lambda}K^\circ$ for every $\lambda>0$. A simple computation using polar coordinates shows that
\begin{equation}
	\label{pre_polares}
	\vol(K)=\frac{1}{n}\int_{\S}r_K^n(y)dy=\frac{1}{n}\int_{\S}\|y\|^{-n}_K dy\, .
\end{equation}

The \BS inequality states that if $K$ is origin-symmetric, then
\begin{equation}
	\label{pre_blaschkesantalo}
	\vol(K) \vol(K^\circ) \leq \vball n^2
\end{equation}
and equality holds if, and only if, $K$ is a centered ellipsoid, where $\vball n$ is the volume of the Euclidean unit ball $\B^n_2 \subset \Rn$ for $n \geq 2$.

The affine term defined in \eqref{def_affineterm} has an interesting geometrical interpretation. To any function $f \in W^{1,p}(\R^n)$ we may associate a norm $\|\cdot\|_{f,p}$ given by
\[\|\xi \|_{f,p} = \|\nabla_\xi  f\|_p = \left(\int_{\Rn} |\langle \nabla f(x), \xi  \rangle|^p dx \right)^{\frac 1p}.\]
The corresponding unit ball will be denoted by $\L$ and its volume, when computed in polar coordinates, gives the identity
\begin{equation}
	\label{eq_relE_Lf}
	\E f = \constE n^{-\frac 1n} \vol(\L)^{-\frac 1n}.
\end{equation}

Assume $p=1$ and $f$ is the characteristic function of a convex body $K$, then $L_{1,f}$ is, up to a constant depending on $n$ and $p$, the polar projection body $\Pi^\circ K$ (see \cite[Definition 10.77]{schneider2014convex}) and inequality \eqref{ineq_Zhang} becomes the \PP inequality, an affine-invariant version of the classical isoperimetric inequality.
The set $\L$ appears in the literature, sometimes with the notation $\Pi_p^\circ f$ (see for example \cite{alonso2018zhang, da2019some}) since it is a functional version of the polar projection operator. For a given convex body $K\subset\R^n$ there are many bodies associated to it. In particular, Lutwak and Zhang introduced in \cite{lutwak1997blaschke} for a convex body $K$ its $\Lp$-centroid body $\Gamma_pK$ defined by

\[
h_{\Gamma_pK}^p(y):=\frac{1}{\constCentroid \vol(K)}\int_{K}|\langle y,z\rangle|^p dz\quad \mbox{ for }y\in\R^n\, ,
\]
where

\[
\constCentroid  = \frac{\vball {n+p}}{\vball 2 \vball n \vball {p-1}}\, .
\]
There are some other normalizations of the $\Lp$-centroid body and the previous one is made so that $\Gamma_p \B=\B$ for the unit ball in $\R^n$ centered at the origin.

The definition of $\Gamma_pK$ can also be written as
\[
	h_{\Gamma_pK}^p(y)=\frac{1}{ \constCentroidPol  \vol(K)}\int_{\S} r_K(\xi)^{n+p}|\langle y,\xi \rangle|^p d\xi\quad \mbox{ for }y\in\R^n,
\]
with $\constCentroidPol = \frac{ n \vball {n+p-2}}{\vball 2 \vball {n-2} \vball {p-1}}$.

Inequalities (usually affine invariant) that compare the volume of a convex body $K$ and that of an associated body are common in the literature. For the specific case of $K$ and $\Gamma_pK$, Lutwak, Yang and Zhang \cite{lutwak2000lp} (see also \cite{campi2002lp} for an alternative proof) came up with what it is known as the $\Lp$ \BP centroid inequality, namely
\begin{equation}
\label{pre_BPineq}
\vol(\Gamma_pK)\geq \vol(K)\, .
\end{equation}
In this inequality, equality holds if, and only if, $K$ is a centrally symmetric ellipsoid.

The Wulff laplacian is defined, for a convex body $K$ containing the origin as interior point, as
\[
\Delta_{p,K} f(x) = -\div\left(\nabla \left( \frac {h_K^p}{p} \right) (\nabla f(x))\right) = -\div\left(h_K(\nabla f(x))^{p-1} \nabla h_K (\nabla f(x))\right).
\]
If $K = \B$ we obtain the usual $p$-Laplacian.
Regarding the affine $p$-Laplace operator, we have the following relation

\[
H_f(v) = h(G_f, v),
\]
where

\[
G_f = {\left( \frac{\vball n}{\vol(\L)}\right)^{1/n}\Gamma_p \L[f]},
\]
and the affine $p$-Laplace operator can be written in terms of the Wulff laplacian

\[
\affplap f = \Delta_{p, G_f} (f).
\]

Given a function $f \in W^{1,p}(\R^n)$, its distribution function $\mu_f:[0,\infty) \to [0,\infty)$ is defined by
\begin{equation}
	\label{def_mu}
	\mu_f(t) = \vol\left(\{x \in \Rn \suchthat |f(x)| > t\}\right).
\end{equation}
If $K$ is a convex body and $f\in \W[\Rn]$, the decreasing rearrangement $f^K:\Rn \to \R$ is the unique function with level sets of the form $\lambda K$ with $\lambda > 0$, and same distribution function as $f$.
When $K = \B$ we obtain the symmetric rearrangement function denoted by $f^*$.
%INCLUIMOS UMA REFERÊNCIA?

Besides, for any measurable set $L \subseteq \R^n$ we denote $L^*$ the closed ball centered at the origin, with same Lebesgue measure as $L$.

The classical \PZ principle states that if $f \in W^{1,p}(\R^n)$, then $f^*$ also belongs to $W^{1,p}(\R^n)$ and

\[
\gnorm f \geq \gnorm{f^*}.
\]
In \cite{nguyen2015new}, Nguyen proved the affine \PZ principle for a general affine operator, together with the corresponding \BZ type result.
We shall need the following particular case.
\begin{proposition}[Theorem 1.1 of \cite{nguyen2015new}]
	\label{thm_affinePZ}
	If $p > 1$ and $f \in W^{1,p}(\R^n)$, then
	\begin{equation}
		\label{eq_affinepzprinciple}
		\E f \geq \E f^*.
	\end{equation}
	Moreover, if $f$ is a non-negative function such that
	
\[
\vol(\{x \in \Rn \suchthat |\nabla f^*| = 0 \hbox{ and } 0 < f^*(x) < \esssup f\} ) = 0,
\]
where $\esssup$ denotes the essential supremum of $f$, then the equality holds in \eqref{eq_affinepzprinciple} if, and only if,

\[
f(x) = f^{\mathbb E}(x+x_0),
\]
where $x_0 \in \Rn$ and ${\mathbb E}$ is an origin symmetric ellipsoid.
\end{proposition}

\section{The affine \Poincare inequality}
\label{sec_affinepoinc}

This section is devoted to establish a powerful result, Theorem \ref{thm_reversezhang}, by using the \BS inequality. This result will be fundamental in the proof of Theorems \ref{thm_affinePI}, \ref{thm_affineRK}, \ref{thm_existence}, \ref{thm_Lambdaproperties} and \ref{thm_equalitycase}, being that the first three ones will be presented still in this section.

\subsection{A reverse comparison inequality in bounded open sets}
\label{sec_existence}

Our central reverse inequality is stated as

\begin{theorem}
	\label{thm_reversezhang}
	Let $\domain \subseteq \Rn$ be a bounded open set and $p \geq 1$. Then, for any $C^1$ function $f :\Rn \to \R$ with support in $\domain$,

\begin{equation}
		\label{ineq_reversezhang}
		\E f \geq \constRZ \|f\|_p^{\frac {n-1}{n}} \gnorm{f }^{1/n},
	\end{equation}
	where
	
\[
\constRZ = \constAbsRZ \left( \max_{\xi \in \S} \w(\domain, \xi) \right)^{-\frac {n-1}{n}},
\]
being $\w(\domain, \xi)$ the width of $\domain$ in the direction $\xi$,	

\[
\constAbsRZ = \constE \left(\frac 2n \frac{\vball {n-1}}{\vball n^2} \constTal^{n-1} \constAuxRZ^{1/p} \right)^{1/n},
\]

\[
\constAuxRZ = \frac 1{n \vball n} \int_\S |\langle e_1, \xi \rangle|^p d\xi = \frac{2 \omega _{n+p-2}}{n \omega _n \omega _{p-1}}
\]
and

\[
\constTal =
	\left\{ \begin{array}{lc}
		\frac 2p (p-1)^{\frac 1p - 1} \frac{\pi-\pi/p}{\sin(\pi/p)},& {\rm if}\ p>1\\
		2,& {\rm if}\ p=1
		\end{array}
	\right.
\]
Moreover, inequality is {\it strict} in the sense that the constant $\constAbsRZ$ is not the best possible, even when $\domain = \B$.
\end{theorem}

In order to prove the above result, we need the following useful lemma:
\begin{lemma}
	\label{lem_pbound}
	Let $\domain \subseteq \Rn$ be a bounded open set, $p \geq 1$ and $f :\Rn \to \R$ be a $C^1$ function with support in $\domain$. For each fixed $\xi \in \S$, we have

\[
\|\nabla_\xi f\|_p \geq \constTal \|f\|_p \w(\domain, \xi)^{-1}
\]
where $\constTal$ and $\w(\domain, \xi)$ is as in the statement of Theorem \ref{thm_reversezhang}.
\end{lemma}
\begin{proof}
	By using the sharp one-dimensional \Poincare inequality given in \cite[pag 357]{talenti1976best} with $q=p$, we get
	\begin{align*}
		\|\nabla_\xi f \|_p^p
		&= \int_{\xi^\bot} \int_{-\infty}^{\infty} \left| \frac{\partial}{\partial t} f(t \xi + x) \right|^p dt dx\\
		&\geq \constTal^p \int_{\xi^\bot} \int_{-\infty}^\infty  |f(t \xi + x)|^p dt\ dx \w(\domain, \xi)^{-p}\\
		&\geq \constTal^p \|f\|_p^p \w(\domain, \xi)^{-p}.
	\end{align*}
\end{proof}

\begin{proof}[Proof of Theorem \ref{thm_reversezhang}]
	Consider the set $\L^\circ$ whose support function is $h(\L^\circ, \xi) = \|\nabla_\xi f \|_p$.
	By Lemma \ref{lem_pbound}, we have

\[
h(\L^\circ, \xi) \geq \constTal \|f\|_p \w(\domain, \xi)^{-1},
\]
so that $\L^\circ$ contains a fixed ball of radius

\[
\frac{\constTal \|f\|_p}{\max_{\xi \in \S} \w(\domain, \xi)}.
\]

	%Now since $|(x_1, \ldots, x_n)|_2^p \leq c \sum |x_i|^p$ we have $\gnorm{f }^p \leq c \sum_i \| \nabla_{e_i} f \|_p^p$
	%and then
	%\[h(\L^\circ, v)^p = \|\nabla_v f\|_p^p \geq c \frac 1n \gnorm{f }^p\] for some $v = e_i \in S^{n-1}$.
Besides,

	\begin{align*}
		 \max_{\xi \in \S} h(\L^\circ, v)^p
		&= \max_{\xi \in \S} \int_{\Rn}|\nabla_\xi f(x)|^p dx\\
		&\geq \frac 1{n \vball n} \int_{\S} \int_{\Rn}|\nabla_\xi f(x)|^p dx d\xi\\
		&= \int_{\Rn}\frac 1{n \vball n} \int_{\S}|\nabla_\xi f(x)|^p d\xi dx\\
		&= \constAuxRZ \int_{\Rn}|\nabla f(x)|_2^p dx,
	\end{align*}
where $\constAuxRZ$ is as in the statement of Theorem \ref{thm_reversezhang}. But this means that $\L^\circ$ contains a point whose distance to the origin is at least $\constAuxRZ^{1/p} \gnorm{f }$.

By the two previous remarks and since $\L^\circ$ is symmetric, we have an entire double cone inside $\L^\circ$ and
	
\[
\vol(\L^\circ) \geq \frac 2n \gnorm f \constAuxRZ^{1/p} \vball {n-1}\left(\frac{ \constTal \|f\|_p}{\max_{\xi \in \S} \w(\domain, \xi)}\right)^{n-1}.
\]
Finally, by the \BS inequality,
		
\[
\E f = \constE \vol(\L)^{-1/n} \geq \vball n^{-2/n} \constE \vol(\L^\circ)^{1/n}
\]
and the result follows with the constant $\constAbsRZ$ of the statement of Theorem \ref{thm_reversezhang}.
	
\end{proof}

\subsection{Proof of Theorems \ref{thm_affinePI}, \ref{thm_affineRK} and \ref{thm_existence}}
Now we are ready for proving the first three theorems of the paper.

\begin{proof}[Proof of Theorem \ref{thm_affinePI}]
It is an immediate consequence of Theorem \ref{thm_reversezhang} and the classical $L^p$ \Poincare inequality \eqref{ineq_classicalpoincare}.
\end{proof}

\begin{proof}[Proof of Theorem \ref{thm_affineRK}]
	Let $(f_k) \subset \W$ be a sequence such that $\E f_k \leq 1$ for all $k \geq 1$. If, for some subsequence, $f_k$ converges to zero in $L^p(\domain)$, then the desired conclusion follows readily. Otherwise, there exists a constant $c > 0$ such that $\|f_k\|_p \geq c$ for all $k \geq 1$. In this case, by Theorem \ref{thm_reversezhang}, we know that $(f_k)$ must be bounded in $\W$. The first assertion then follows from the classical \RK theorem.

For the second part, we construct sequences of functions $(f_k)$ such that $\E f_k = 1$ and $\gnorm{f_k} \to + \infty$. Indeed, it suffices to obtain sequences $(f_k)$ that satisfy $\frac{\gnorm{f_k}}{\E f_k} \to + \infty$. We consider two cases:

\underline {The case $p=1$:}
	
Let $f_k$ be the characteristic function of the set $[0,1]^{n-1} \times [0,1/k]$ and assume without loss of generality that this set is inside $\domain$. Using the invariance property of $\E[1]$, we have
	
\begin{eqnarray*}
		\gnorm[1]{f_k} &=& 2 + (\gnorm[1]{f_1} - 2)/k,\\
		\E[1]f_k &=& k^{-\frac{n-1}n} \E[1] f_1,
\end{eqnarray*}
so that $\frac {\gnorm[1]{f_k}}{\E[1]f_k} \to +\infty$.

	\underline {The case $p>1$:}

Let $J \subset \R$ be an interval and $U \subset \Rn$ be an open set such that $J \times U \subseteq \domain$. Assuming without loss of generality that $J = [0,1]$, we define for each $k = 4,5, \ldots$, the set $A_k = [\frac 13 - \frac 1k, \frac 13] \cup [\frac 23, \frac 23+ \frac 1k]$ and the function $\phi_k: [0,1] \to [0,1]$ as
	
\begin{equation}
		\phi_k(x) = \left \{
			\begin{array}{cc}
				1 + \frac k6 - |x-\frac 12|k, & x \in A_k \\
				1, & x \in [\frac 13, \frac 23]\\
				0, & \hbox{ otherwise }
			\end{array}
			\right.
\end{equation}
Take $\eta : U \to \R$ a smooth compact-supported function and define for $(t,x) \in \R \times \R^{n-1}$,
	
\[
f_k(t,x) = \phi_k(t) \eta(x).
\]
For $\xi = (a,v) \in \R \times \R^{n-1}$, we make the following computations
	
\[
		\nabla_\xi f_k(t,x) = (a \phi_k'(t) \eta(x),  \phi(t) \nabla_v \eta(x)),
\]

\[
\int_\domain |\nabla_\xi f_k(x)|^p dx \leq C(a^p k^{p-1} + |v|^p)
\]
and

\begin{align}
		\E^{-n} f
		&\geq C \int_{\S} (a^p k^{p-1} + 1)^{-n/p} d\xi\\
		%&= C \int_{-1}^1 (a^p k^{p-1} + 1)^{-n/p} (1-t^2)^{\frac{n-3}2} da\\
		%&= C \int_{-\pi/2}^{\pi/2} (\sin(\alpha)^p k^{p-1} + 1)^{-n/p} \cos(\alpha)^{n-2} \cos(\alpha) d\alpha\\
		&= C \int_{-\pi/2}^{\pi/2} (\sin(\alpha)^p k^{p-1} + 1)^{-n/p} \cos(\alpha)^{n-2} d\alpha\\
		&\geq C \int_0^{\pi/4} (\sin(\alpha)^p k^{p-1} + 1)^{-n/p} d\alpha\\
		&\geq C \int_0^{\pi/4} (\alpha^p k^{p-1} + 1)^{-n/p} d\alpha\\
		&\geq C k^{-\frac{p-1}p} \int_0^{\frac{\pi}{4} k^{\frac{p-1}p}} (\beta^p + 1)^{-n/p} d\beta\\
		&\geq C k^{-\frac{p-1}p},
\end{align}
so that $\E f \leq C k^{\frac 1n \frac{p-1}p}$.
	
On the other hand,
	
\[
\int_\domain |\nabla f_k(x)|^p dx \geq \int_{A_k}\phi_k'(t)^p dt \int_U \eta(x)^p + \int_J \phi_k(t)^p dt \int_U |\nabla \eta(x)|^p dx \geq C k^{p-1}
\]
and we obtain

\[
\frac{\gnorm{f_k}}{\E f_k} \geq C \frac{ k^\frac{p-1}p}{k^{\frac 1n \frac{p-1}p}} \to + \infty.
\]

	Notice that in both examples we have the same divergence rate in both sides of inequality \eqref{ineq_reversezhang}, showing that this inequality is asymptotically sharp.
\end{proof}

\begin{proof}[Proof of Theorem \ref{thm_existence}]
	Let $(f_k) \subset \W$ be a sequence such that $\|f_k\|_p = 1$ and $\E^p f_k \to \affineeigenvalue$. Then, by Theorem \ref{thm_reversezhang}, $(f_k)$ must be bounded in $\W$.

Assume $p > 1$. By reflexivity and \RK theorem, up to a subsequence, we have $f_k \rightharpoonup f$ in $\W$ and $f_k \to f$ in $L^p(\domain)$.
	Clearly, $\|f\|_p = 1$. Now we claim that
	
\[
\E^p f \leq \liminf_{k \rightarrow + \infty} \E^p f_k = \affineeigenvalue
\]
which proves the theorem.

	For each fixed $\xi \in \S$, we have $\nabla_\xi f_k \rightharpoonup \nabla_\xi f$ in $L^p(\domain)$, so that

\[
\liminf_{k \rightarrow + \infty} \|\nabla_\xi f_k\|_p \geq \|\nabla_\xi f\|_p.
\]
By Lemma \ref{lem_pbound} and $\|f\|_p = 1$, we know that $\|\nabla_\xi f_k\|_p \geq c > 0$ for all $k \geq 1$. Thus, Fatou's lemma gives
	\begin{align*}
		\int_{\S} \|\nabla_\xi f\|_p^{-n} d\xi
		&\geq \int_{\S} \limsup_{k \rightarrow + \infty} \|\nabla_\xi f_k\|_p^{-n} d\xi\\
		&\geq \limsup_{k \rightarrow + \infty} \int_{\S} \|\nabla_\xi f_k\|_p^{-n} d\xi
	\end{align*}
	and the claim follows.

	For $p=1$, the same reasoning yields a sequence $(f_k) \subset W_0^{1,1}(\domain)$ converging in the $L^1$ topology to $f \in L^1(\domain)$.
	The fact that
	
\[
\liminf_{k \rightarrow + \infty} \|\nabla_\xi f_k\|_1 \geq \|\nabla_\xi f\|_1
\]
follows easily from the usual definition of $\|\nabla_\xi f\|_1$ for functions in $BV(\domain)$ (see \cite{evans2015measure}).
\end{proof}

\section{Properties of extremals (minimizers)}
\label{sec_ELandregularity}

This section is dedicated to fine differential analysis of extremals related to the affine $L^p$ \Poincare inequalities for $p > 1$. The nexus between the affine invariant $\affineeigenvalue$ and the PDEs setting for its extremals is presented below by mean of an Euler-Lagrange equation satisfied by minimum points of the affine Rayleigh $p$-quotient $R^{{\mathcal A}}_p(f)$. In particular, it will appear the affine $p$-Laplace operator $\affplap$ mentioned in the introduction. The remaining sections focus on invariance and regularity properties satisfied by such an operator and the proof of Theorem \ref{thm_ELZhang}.

\subsection{The affine \EL equation and associated operator}

The following result is a crucial ingredient for the rest of the paper:

\begin{theorem} \label{thm_ELequation}
	Let $\domain \subset \R^n$ be a bounded open set and $p > 1$. Let also $f := f^{\mathcal A}_p \in \W$ be a minimizer associated to $\affineeigenvalue$ whose existence is guaranteed by Theorem \ref{thm_existence}. Then, $f$ satisfies

\[
\int_\domain \langle H_{f}^{p-1}(\nabla f) \nabla H_{f}(\nabla f), \nabla \psi \rangle dx = \affineeigenvalue \int_\domain |f|^{p-2} f \psi dx
\]
for every function $\psi \in \W$. In PDEs language, $f$ is said to be a weak solution of the Dirichlet problem

\[
\left\{
\begin{array}{rrll}
\affplap f &=&  \affineeigenvalue |f|^{p-2}f & {\rm in} \ \ \domain,\\
	f &=& 0  & {\rm on} \ \ \boundary{\domain},
\end{array}
\right.
\]
and also an eigenfunction of the operator $\affplap$ on $\W$ corresponding to its first eigenvalue $\affineeigenvalue$.
\end{theorem}

\begin{proof}
	Let $f \in \W$ be a minimizer of \eqref{def_affineeigenvalue} so that for every $g \in \W$,
	
\[
{\int_{\S} \|\nabla_\xi g\|_{L^p}^{-n} d\xi} - \constE^n \affineeigenvalue^{-n/p} {\|g\|_p^{-n}} \leq 0
\]
with equality if $g=f$.

Take any fixed $\psi \in \W$ and compute the derivative on the left-hand side with respect to $\psi$.
It means to replace $g$ by $f + \varepsilon \psi$ and take derivative with respect to $\varepsilon$.

\begin{align*}
	\frac \partial {\partial \psi} & \left( \int_{\S} \|\nabla_\xi f\|_{L^p}^{-n} d\xi \right)\\
	&= -\frac np \int_{\S} \left( \int_{\Rn} |\nabla_\xi f (x)|^p dx \right)^{-\frac np -1} \int_{\Rn} p |\nabla_\xi f (x)|^{p-1} \sg(\nabla_\xi f) \nabla_\xi \psi (x) dx d \xi\\
	&= -n \int_{\S}  \left( \int_{\Rn} |\nabla_\xi f (y)|^p dy \right)^{-\frac n p -1} \int_{\Rn} \{\nabla_\xi f (x)\}^{p-1} \nabla_\xi \psi (x) dx d \xi\\
\end{align*}
where $\{x\}^p := |x|^p \sg(x)$.

Applying Fubini's theorem,

\begin{align*}
	\frac \partial {\partial \psi} & \left( \int_{\S} \|\nabla_\xi f\|_{L^p}^{-n} d\xi \right)\\
	&= -n \int_{\Rn} \int_{\S} \|\nabla_\xi f\|_p^{-n-p}\{\nabla_\xi f (x)\}^{p-1} \langle \nabla \psi (x), \xi \rangle d \xi dx\\
	&= -n \int_{\Rn} \langle \nabla \psi (x),\int_{\S} \|\nabla_\xi f\|_p^{-n-p}\{\nabla_\xi f (x)\}^{p-1} \xi d \xi \rangle dx.\\
\end{align*}

Now consider the convex body $K_f$ with support function

\[
h_{K_f}(z)^p = \int_{\S} \|\nabla_\xi f\|_p^{-n-p} |\langle z, \xi \rangle|^p d \xi.
\]
A straightforward computation gives

\[
\nabla \left( \frac 1p h_{K_f}^p \right)(z) = \int_{\S} \|\nabla_\xi f\|_p^{-n-p} \{\langle z, \xi \rangle\}^{p-1} \xi d \xi,
\]
thus we obtain

\[
\frac \partial {\partial \psi} \left( \int_{\S} \|\nabla_\xi f\|_{L^p}^{-n} d\xi \right) = n \int_{\Rn} \left\langle \nabla \psi (x),  \nabla \left( \frac 1p h_{K_f}^p \right)(\nabla f(x)) \right\rangle dx.
\]

The derivative on the right-hand side is

\begin{align*}
	\frac \partial {\partial \psi} \left[ \left( \int_{\Rn} |f(x)|^p dx \right)^{-\frac np} \right]
	&= -\frac np \left( \int_{\Rn} |f(x)|^p dx \right)^{-\frac np-1} \int_{\Rn} p\{f(x)\}^{p-1} \psi(x) dx\\
	&= -n \|f\|_p^{-(n+p)} \int_{\Rn} \{f(x)\}^{p-1} \psi(x) dx
\end{align*}
and we derive the weak formulation of \eqref{eq_ELZhang} for $f \in \W$, that is
	\begin{equation}
		\label{eq_ELZhangWeak}
		\int_{\Rn} \left\langle \nabla \psi (x),  \nabla \left( \frac 1p h_{K_f}^p \right)(\nabla f(x)) \right\rangle dx - \constE^n \affineeigenvalue^{-n/p} \|f\|_p^{-(n+p)} \int \{f(x)\}^{p-1} \psi(x) dx = 0
	\end{equation}
for all $\psi \in \W$.

If $f$ is a minimizer of $C^2$ class on $\overline{\domain}$, then integration by parts yields
\begin{align*}
	        \frac \partial {\partial \psi} & \left( \int_{\S} \|\nabla_\xi f\|_{L^p}^{-n} d\xi \right)\\
		&= n \int_{\Rn} \psi (x) \div \left( \nabla \left( \frac 1p h_{K_f}^p \right)(\nabla f(x)) \right) dx\\
		&= n \int_{\Rn} \psi (x) \Delta_{K_f}^p f(x) dx
\end{align*}
and since \eqref{eq_ELZhangWeak} holds for every $\psi \in \W$, we obtain the (classical) equation

\[
-\Delta_{K_f}^p f(x) + \constE^n \affineeigenvalue^{-n/p} \|f\|_p^{-(n+p)} |f(x)|^{p-2} f(x) = 0
\]

Finally, multiplying the above equation by
			
\[
\left( \frac{\vol(\L)}{\vball n} \right)^{-\frac pn} \frac 1{\constCentroidPol \vol(\L)}  = \constE^{-n} \E^{n+p} f
\]
where the value of the constant $\constCentroidPol$ is given in Section 2, and using the equality $\affineeigenvalue = \frac {\E^p f}{\|f\|_p^p}$, we get

\[
-\affplap f(x) + c_{n,p}^{-n} \constE^n \affineeigenvalue^{-n/p} \left( \frac{\E f}{\|f\|_p}\right)^{n+p} |f(x)|^{p-2} f(x) = 0,
\]
so that

\[
-\affplap f(x) + \affineeigenvalue |f(x)|^{p-2} f(x) = 0\ \ {\rm in}\ \ \domain.
\]

Conversely, if $f \in C^2(\overline{\domain})$ is a solution of \eqref{eq_ELZhang}, then taking $\psi = f$ as a test function, we obtain in each term

\begin{align*}
       \int_{\Rn} \left\langle \nabla f (x),  \nabla \left( \frac 1p H_f^p \right)(\nabla f(x)) \right\rangle dx
		&= \constE^{-n} \E^{n+p} f \int_{\Rn} \left\langle \nabla f (x), \int_{\S} \|\nabla_\xi f\|_p^{-n-p} \{\nabla_\xi f(x)\}^{p-1} \xi d \xi \right\rangle dx \\
		&= \constE^{-n} \E^{n+p} f \int_{\S} \|\nabla_\xi f\|_p^{-n-p} \int_{\Rn} \{\nabla_\xi f(x)\}^{p-1} \left\langle \nabla f (x), \xi \right\rangle dx d \xi \\
		&= \constE^{-n} \E^{n+p} f \int_{\S} \|\nabla_\xi f\|_p^{-n-p} \int_{\Rn} |\nabla_\xi f(x)|^{p} dx d \xi \\
		&= \constE^{-n} \E^{n+p} f \int_{\S} \|\nabla_\xi f\|_p^{-n} d \xi\\
		&= \E^{p} f
\end{align*}
and

\[
\affineeigenvalue \int_{\Rn} |f(x)|^{p-2} f(x) f(x) dx = \affineeigenvalue \|f\|_p^p,
\]	
thus $f$ must be a minimizer.
\end{proof}

\subsection{Affine invariance properties}
Here we deal with the invariance properties of $\affplap$.

\begin{proposition} \label{invariance}

Let $A \in \sl$, $f \in W^{1,p}(\Rn)$, $\lambda > 0$ and $K \subset \Rn$ be a convex body. Denote $f_A(x) = f(Ax)$ for $x \in \Rn$. Then:

\begin{itemize}

\item[(a)] $\Delta_{p,AK} f(x) = (\Delta_{p,K} f_A)(A^{-1}x)$,

\item[(b)] $\affplap f_A(x) = (\affplap f)(Ax)$,

\item[(c)] $\affplap (\lambda f) (x) = \lambda^{p-1} \affplap f(x)$.
\end{itemize}
\end{proposition}

The proof of this result relies on the following trivial facts:
\begin{lemma} \label{invarset}
Let $A$, $f$, $f_A$ and $K$ be as above. Then:

\begin{itemize}
		
\item[(i)] $\Gamma_p (AK) = A \Gamma_p K$,

\item[(ii)] $\L[f_A] = A^T \L$,

\item[(iii)] $\L[\lambda f] = \lambda^{-1} \L$,
	
\end{itemize}
where $A^T$ denotes the transposed matrix of $A$.
\end{lemma}
\begin{proof}[Proof of Proposition \ref{invariance}]
	For the proof of the claim (a), we use (i) of Lemma \ref{invarset} in the computation
	
\begin{align*}
		\Delta_{p,{AK}} f(x)
		&= -\div\left( \nabla \left( \frac{h_{A K}^p}{p} \right)(\nabla f(x)) \right)\\
		&= -\div_x\left( \nabla_\xi \left( \frac{h_K(A^T \xi)^p}{p} \right)(\nabla f(x)) \right)\\
		&= -\div\left(A \nabla \left( \frac{h_{K}^p}{p} \right)(A^T \nabla f(x)) \right)\\
		&= -\div\left(A \nabla \left( \frac{h_{K}^p}{p} \right)(\nabla f_A(A^{-1} x)) \right)\\
		&= (\Delta_{p,K} f_A)(A^{-1} x),
\end{align*}
where we used the identity $\div (A V(A^{-1} x) ) = \div(V)(A^{-1}x)$ valid for any smooth vector field $V$ in $\Rn$.

For the proof of the claim (b), recall that
	
\[
G_f = {\left( \frac{\vball n}{\vol(\L)}\right)^{1/n}\Gamma_p \L[f]}.
\]
By (ii) of Lemma \ref{invarset}, we then get
	
\begin{align*}
		G_{f_A}
		&= {\left( \frac{\vball n}{\vol(\L[f_A])}\right)^{1/n}\Gamma_p \L[f_A]} = {\left( \frac{\vball n}{\vol(A^T \L)}\right)^{1/n}\Gamma_p A^{-1} \L[f]}\\
		&= {\left( \frac{\vball n}{\vol(\L)}\right)^{1/n} A^{-1} \Gamma_p \L[f]} = A^{-1} G_f,
\end{align*}
so that

\[
\affplap f_A(x) = \Delta_{p,G_{f_A}} f_A(x) = \Delta_{p,{A^{-1} G_f}} f_A(x) = (\affplap f)(A x).
\]

Finally, the claim (c) follows from the computation
\[
G_{\lambda f} = {\left( \frac{\vball n}{\vol(\L[\lambda f])}\right)^{1/n}\Gamma_p \L[\lambda f]} = {\left( \frac{\vball n}{\lambda^{-n} \vol(\L[\lambda f])}\right)^{1/n} \lambda^{-1} \Gamma_p \L[\lambda f]} = G_f,
\]
where it was used (iii) of Lemma \ref{invarset}, and

\[
\affplap {\lambda f}(x) = \Delta_{p,G_{{\lambda f}}} {\lambda f}(x) = \Delta_{p,{G_f}} {\lambda f}(x) = \lambda^{p-1} (\affplap f)(x).
\]
\end{proof}

An important property satisfied by $\affplap$ which will be invoked later is

\begin{proposition}
For any radial function $f \in W^{1,p}(\Rn)$, we have

\[
\affplap f = \Delta_p f\ \ {\rm in}\ \ \Rn.
\]
\end{proposition}
\begin{proof}
	Since $f$ is radial, clearly $\L = r \B$ is a ball and
	\[
		G_f = {\left( \frac{\vball n}{\vol(r\B)}\right)^{1/n}\Gamma_p (r\B)} = \B.
	\]
	In this case, $\affplap f = \Delta_{p,G_f} f = \Delta_{p,\B} f$ in $\Rn$, where the latter is the usual $p$-Laplace operator.
\end{proof}

\subsection{Regularity properties and proof of Theorem \ref{thm_ELZhang}}

From the point of view of PDEs analysis, it is essential to know some regularization fine theory satisfied by $\affplap$. The next result collects some of its main smoothing properties.

\begin{proposition} \label{regularity}
	Let $\domain \subseteq \R^n$ be a bounded open set and $p > 1$. Let $f_0 \in \W$ be a weak solution of the problem

 \begin{equation} \label{P1}
\left\{
\begin{array}{rrll}
\affplap f_0 &=& h_0 & {\rm in} \ \ \domain, \\

	f_0 &=& 0 & {\rm on} \ \ \boundary{\domain},
\end{array}\right.
\end{equation}
where $h_0: \domain \rightarrow \R$ is a measurable function. Then, it holds that:

\begin{itemize}
\item[(a)] if $p < n$ and $h_0 \in L^{n/p}(\domain)$, then $f_0 \in L^s(\domain)$ for every $s \geq 1$;
\item[(b)] if $p \leq n$ and $h_0 \in L^q(\domain)$ for some $q > n/p$, then $f_0 \in L^\infty(\domain)$;
\item[(c)] if $h_0 \in L^\infty(\domain)$, then $f_0 \in C^{1, \alpha}(\domain)$ for some $0 < \alpha < 1$;
\item[(d)] if $h_0 \in L^\infty(\domain)$, then $f_0 \in C^{1, \alpha}(\overline{\domain})$ for some $0 < \alpha < 1$ provided that $\partial \domain$ is of $C^{2,\alpha}$ class.
\end{itemize}	
\end{proposition}

\begin{proof}
It suffices to assume $f_0 \neq 0$. Consider the field $a = (a_1, \ldots, a_n) : \R^n \rightarrow \R^n$ given by $a(v) = H_{f_0}^{p-1}(v) \nabla H_{f_0}(v)$ and the related differential operator ${\mathcal L} = {\mathcal L}_{p, f_0}$ on $\W$ defined by

\[
{\mathcal L} f = {\mathcal L}_{p, f_0} f := -{\rm div} \, a(\nabla f).
\]	
Note that ${\mathcal L} f_0$ is a multiple of $\affplap f_0$ by a positive number. Using that $f_0$ is a nonzero function, we show below that ${\mathcal L}$ satisfies the well-known Tolksdorf's structural conditions for quasilinear elliptic operators \cite{tolksdorf1984regularity}:

\begin{itemize}
\item[(T.1)] $\sum^n_{i,j=1} \frac{\partial a_i(v)}{\partial v_j} \eta_i \eta_j \geq C_1 |v|^{p-2} |\eta|^2$;
\item[(T.2)] $\sum^n_{i,j=1} \left| \frac{\partial a_i(v)}{\partial v_j} \right| \leq C_2 |v|^{p-2}$
\end{itemize}
for every $v \in \R^n \setminus \{0\}$ and $\eta \in \R^n$, where $C_1$ and $C_2$ are positive constants independent of $v$ and $\eta$. In particular, ${\mathcal L}$ is a Wulff type degenerate quasilinear elliptic operator.

In fact, by Lemma \ref{lem_pbound} and Cauchy-Schwarz inequality, we derive the lower and upper estimates

\[
D \|f_0\|_p \leq \|\nabla_\xi f_0\|_p \leq \|\nabla f_0\|_p
\]
for every $\xi \in \S$, where $D$ is positive constant. Under the above inequalities, Haberl and Schuster proved (see Lemma 4.1 of \cite{haberl2009general}) that the function $H_{f_0}$ belongs to $C^1(\R^n) \cap C^2(\R^n \setminus \{0\})$. Notice also that, for any $v \in \R^n$,

\[
a_i(v) = \frac{\partial}{\partial v_i} \left( \frac1p H^p_{f_0}(v) \right) = \int_{\S} \|\nabla_\xi f_0\|^{-n-p}_p |\langle \xi, v \rangle|^{p-2} \langle \xi, v \rangle \xi_i d\xi,
\]
and for $v \in \R^n \setminus \{0\}$,

\[
\frac{\partial a_i(v)}{\partial v_j} = (p-1) \int_{\S} \|\nabla_\xi f_0\|^{-n-p}_p |\langle \xi, v \rangle|^{p-2} \xi_i \xi_j d\xi.
\]
As a direct consequence of the latter, for any $v \in \R^n \setminus \{0\}$ and $\eta \in \R^n$, we have

\[
\sum^n_{i,j=1} \frac{\partial a_i(v)}{\partial v_j} \eta_i \eta_j = (p-1) \int_{\S} \|\nabla_\xi f_0\|^{-n-p}_p |\langle \xi, v \rangle|^{p-2} \langle \xi, \eta \rangle^2 d\xi.
\]
Plugging the inequality $\|\nabla_\xi f_0\|_p \leq \|\nabla f_0\|_p$ on the above right-hand side, we derive

\begin{eqnarray*}
\sum^n_{i,j=1} \frac{\partial a_i(v)}{\partial v_j} \eta_i \eta_j &\geq& (p-1) \|\nabla f_0\|^{-n-p}_p \int_{\S} |\langle \xi, v \rangle|^{p-2} \langle \xi, \eta \rangle^2 d\xi \\
&\geq& C_1 |v|^{p-2} |\eta|^2,
\end{eqnarray*}
where

\[
C_1:= (p-1) \|\nabla f_0\|^{-n-p}_p \min_{v \in \S} \int_{\S} |\langle \xi, v \rangle|^{p-2} |\xi_1|^2 d\xi.
\]
Note that $C_1$ is finite and positive, since $p > 1$. This proves the condition (T.1).

For the proof of the condition (T.2), by using the lower estimate $\|\nabla_\xi f_0\|_p \geq D \|f_0\|_p$ for every $\xi \in \S$, we get

\begin{eqnarray*}
\sum^n_{i,j=1} \left| \frac{\partial a_i(v)}{\partial v_j} \right| &\leq& (p-1) \sum^n_{i,j=1} \int_{\S} \|\nabla_\xi f_0\|^{-n-p}_p |\langle \xi, v \rangle|^{p-2} |\xi_i| |\xi_j| d\xi \\
&\leq& n(p-1) \int_{\S} \|\nabla_\xi f_0\|^{-n-p}_p |\langle \xi, v \rangle|^{p-2} d\xi \\
&\leq& n(p-1) D^{-n-p} \|f_0\|^{-n-p}_p \int_{\S} |\langle \xi, v \rangle|^{p-2} d\xi \\
&=& C_2 |v|^{p-2},
\end{eqnarray*}
where

\[
C_2:= n(p-1) D^{-n-p} \|f_0\|^{-n-p}_p \int_{\S} |\xi_1|^{p-2} d\xi.
\]
Again, once $p > 1$, the constant $C_2$ is finite and positive.

Now, thanks to (T.1) and (T.2), the proof of the claims (a) and (b) follows from arguments based on the De Giorgi-Nash-Moser's iterative scheme, developed by De Giorgi \cite{de1960sulla} for elliptic equations and, independently, by Moser \cite{moser1960new} and Nash \cite{nash1957parabolic} for parabolic equations, see \cite{lieberman1988boundary} for more details. Finally, the claims (c) and (d) follow from (b) and $C^{1,\alpha}$ regularity results of \cite{dibenedetto1982c, lieberman1988boundary, tolksdorf1984regularity}. This ends the proof.
\end{proof}

The proof of the smoothness of minimizers bases on Proposition \ref{regularity} and the following result:

\begin{proposition} \label{smoothness}
Let $\domain \subseteq \R^n$ be a bounded open set and $p > 1$. Let $f_0 \in \W$ be a weak solution of the problem

 \begin{equation} \label{P2}
\left\{
\begin{array}{rrll}
\affplap f_0 &=& \rho(x) |f_0|^{p-2} f_0 & {\rm in} \ \ \domain, \\
	f_0 &=& 0 & {\rm on} \ \ \boundary{\domain},
\end{array}\right.
\end{equation}
where $\rho: \domain \rightarrow \R$ is a weight function. If $\rho \in L^{n/p}(\domain)$, then $f_0 \in L^s(\domain)$ for every $s \geq 1$.
\end{proposition}

\begin{proof}
Using the strategy of the previous proof which consists in introducing the quasilinear elliptic operator ${\mathcal L} = {\mathcal L}_{p, f_0}$ on $\W$ for fixed $f_0$, we rewrite \eqref{P2} as

\[
\left\{
\begin{array}{rrll}
{\mathcal L} f_0 &=& \rho(x) |f_0|^{p-2} f_0 & {\rm in} \ \ \domain, \\
	f_0 &=& 0 & {\rm on} \ \ \boundary{\domain},
\end{array}\right.
\]
where ${\mathcal L}$ satisfies the conditions (T.1) and (T.2) as already before proved. By applying now Proposition 1.2 of \cite{guedda1989quasilinear}, the conclusion of the statement follows.

\end{proof}

We conclude this subsection with the

\begin{proof}[Proof of Theorem \ref{thm_ELZhang}]

In Theorem \ref{thm_ELequation}, it was shown that a function $f^{\mathcal A}_p \in \W$ minimizes of the best \Poincare constant $\affineeigenvalue$ if, and only if, it is a weak solution of the eigenvalue problem

\[
\left\{
\begin{array}{rrll}
\affplap f &=& \affineeigenvalue |f|^{p-2} f & {\rm in} \ \ \domain, \\
	f &=& 0 & {\rm on} \ \ \boundary{\domain},
\end{array}\right.
\]
This easily implies that $\affineeigenvalue$ is the smallest among all eigenvalues of the operator $\affplap$ on $\W$.

Now we establish the smoothness of the minimizer $f_0 = f^{\mathcal A}_p$. By the above observation, $f_0$ is a weak solution of \eqref{P2} in $\W$ when we take $\rho$ as the constant function $\rho(x) = \affineeigenvalue$. Then, by Proposition \ref{smoothness}, we deduce that $f_0 \in \bigcap_{s \geq 1} L^s(\domain)$. Choosing then $h_0 = \rho(x) |f_0|^{p-2} f_0 = \affineeigenvalue |f_0|^{p-2} f_0$ in \eqref{P2}, by Proposition \ref{regularity}, we conclude that $f_0$ is a bounded function in $C^{1,\alpha}(\domain)$ for arbitrary $\domain$ and in $C^{1,\alpha}(\overline{\domain})$ if $\domain$ has boundary of $C^{2,\alpha}$ class.

Finally, the above smoothness property is crucial in showing that the minimizer $f_0 = f^{\mathcal A}_p$ has defined sign. In fact, since $R^{{\mathcal A}}_p(f) = R^{{\mathcal A}}_p(|f|)$ for every $f \in \W$, then $f_0 = |f^{\mathcal A}_p|$ belongs to $\W$ and is also a minimizer (or eigenfunction) of $\affineeigenvalue$. Thus, by the previous conclusion, $f_0$ is a non-negative function in $C^{1,\alpha}(\domain)$ such that ${\mathcal L} f_0 = \affineeigenvalue f_0^{p-1} \geq 0$ in $\domain$ in the weak sense, where ${\mathcal L} = {\mathcal L}_{p, f_0}$ is defined in the proof of Proposition \ref{regularity}. Invoking the strong maximum principle for $C^1$ super-solutions of quasilinear elliptic equations involving operators of type ${\mathcal L}$ (see for example \cite{pucci2007j}), one concludes that $f_0 > 0$ in $\domain$ since $f_0$ is nonzero, and so we complete the proof. 	
\end{proof}

\section{The affine \FK inequality}
\label{sec_affineFK}

This section is devoted to the proof of the affine version of the \FK inequality stated in Theorem \ref{thm_affineFK}.

\begin{proof}[Proof of Theorem \ref{thm_affineFK}]
	
We recall here that $\Omega^*$ denotes the closed ball centered at the origin with same Lebesgue measure as $\Omega$. We divide the proof into two cases.

\underline {The case $p=1$:}

We begin with a simple computation. Take any $r > 0$ and let $\chi_{r \B}$ be the characteristic function of the ball $r \B$. Define
	
\[
k_r = \frac {\gnorm[1]{\chi_{r \B}}}{\|\chi_{r \B}\|_1} = \frac {S(r \B)}{\vol(r \B)}
\]
in the $\BV[r \B]$ sense, where $S$ denotes surface area. Clearly, we have

\[
k_r = \frac nr = \frac{n \vball n^{\frac 1n}}{\vol(r \B)^{\frac 1n}},
\]
and since $\chi_{r \B}$ can be approximated by smooth functions with compact support inside $r \B$ where the Rayleigh quotient converges to $k_r$, we deduce that
	
\[
\affineeigenone[\domain^*] \leq \classicaleigenone[\domain^*] \leq \frac{n \vball n^{\frac 1n}}{\vol(\domain^*)^{\frac 1n}}.
\]
Now let $f \in \BV$ be a minimizer of $\affineeigenone$. By \Holder's inequality and the \SZ inequality for $p=1$, we have

\begin{equation}
		\label{ineq_FKp1}
		\affineeigenone = \frac{\E[1] f}{\|f\|_1} \geq \frac{\E[1] f}{\|f\|_{\frac n{n-1}} \vol(\domain)^{\frac 1n}} \geq \frac{n \vball n^{\frac 1n}} {\vol(\domain)^{\frac 1n}} = \frac{n \vball n^{\frac 1n}} { \vol(\domain^*)^{\frac 1n}} \geq \affineeigenone[\domain^*].
\end{equation}
For the equality case, if $\affineeigenone = \affineeigenone[\domain^*]$ we have equality in all inequalities of \eqref{ineq_FKp1}. In particular, the equality case of the \SZ inequality implies that $f$ is the characteristic function of an ellipsoid $\mathbb E$. Finally, the equality case for the \Holder inequality implies $\domain = \mathbb E$.

\underline {The case $p>1$:}

We consider the affine \PZ principle (Proposition \ref{thm_affinePZ} in the section 2). Let $f \in \W$ be a minimizer of $\affineeigenvalue$. Then, $f^* \in \W[\domain^*]$ and
	
\[
\affineeigenvalue = \frac { \E^p f }{ \|f\|_p^p } \geq \frac { \E^p f^* }{ \|f^*\|_p^p } \geq \affineeigenvalue[\domain^*].
\]
For the equality case, if $\affineeigenvalue = \affineeigenvalue[\domain^*]$ then the above inequalities become equalities, so that $f^*$ is a minimizer of $\affineeigenvalue[\domain^*]$. Since $f^*$ is radial, we have $\E f^* = \|\nabla f^*\|_p$ and
	
\[
\classicaleigenvalue[\domain^*] \geq \affineeigenvalue[\domain^*] = \frac{\E^p f^*}{\|f^*\|_p^p} = \frac{\gnorm{f^*}^p}{\|f^*\|_p^p},
\]
so that $f^*$ is also a minimizer of $\classicaleigenvalue[\domain^*]$. Since $f^*$ is a first positive eigenfunction of the $p$-Laplace operator on the ball $\domain^*$, then $f^*$ is radially strictly decreasing. This fact along with the equality $\E f = \E f^*$, by the \BZ theorem in Proposition \ref{thm_affinePZ} applied to $f$, implies that $f(x) = F(|Ax|_2)$ for some smooth function $F: \R \rightarrow \R$ and an invertible matrix $A$.

It remains to show that $A\domain$ is a ball. Let us assume that $\det(A) = 1$ and set $f_A(x) = f(A^{-1}x)$. Then, $f_A$ is radial, so
	
\[
\frac {\gnorm{f_A}^p}{\|f_A\|_p^p} = \frac {\E^p f_A}{\|f_A\|_p^p} = \affineeigenvalue[A \domain] \leq \classicaleigenvalue[A \domain],
\]
thus $f_A$ is a minimizer of $\classicaleigenvalue[A \domain]$. Besides, since $f^*$ is a minimizer of $\classicaleigenvalue[\domain^*]$, we have
	
\[
\classicaleigenvalue[A \domain] = \frac {\E^p f_A}{\|f_A\|_p^p} = \frac {\E^p f^*}{\|f^*\|_p^p} = \frac {\gnorm{f^*}^p}{\|f^*\|_p^p} = \classicaleigenvalue[\domain^*].
\]
Finally, the \FK inequality for the $p$-Laplace operator shows that $A \domain$ must be a ball. This concludes the proof.
\end{proof}

\section{Comparison of eigenvalues}
\label{sec_comparison}

In this section we prove Theorems \ref{thm_Lambdaproperties} and \ref{thm_equalitycase} that relate the affine and classical eigenvalues $\affineeigenvalue$ and $\classicaleigenvalue$.
Lastly, we prove Theorem \ref{thm_affineCheegerSet} showing the existence of affine Cheeger sets.

For the proof of Theorem \ref{thm_Lambdaproperties} we recall an interesting comparison result proved by Huang and Li.
\begin{proposition}[Theorem 1.2 of \cite{huang2016optimal}]
	\label{thm_chino}
Let $p \geq 1$. For any $f \in W^{1,p}(\R^n)$, we have
	
\[
\constchino \min_{A \in \sl} \gnorm{f_A} \leq \E f,
\]
where
	
\[
\constchino = \frac{\pi^{\frac 1{2p} + \frac 12} \Gamma(\frac{n+p}2)^\frac 1p \Gamma( 1 + \frac np)^{\frac 1n}}{2^{\frac 1p + 1} \Gamma(1+\frac n2)^{\frac 1n + \frac 1p} \Gamma(\frac{p+1}2)^{\frac 1p} \Gamma(1+\frac 1p)}.
\]
	
\end{proposition}

\begin{proof}[Proof of Theorem \ref{thm_Lambdaproperties}]

For the proof of (a), take a minimizer $f \in \W$ of $\classicaleigenvalue$. The inequality $\E f \leq \gnorm{f}$ then implies
	
\[
\classicaleigenvalue = \frac {\gnorm{f}^p }{ \|f\|_p^p } \geq \frac {\E^p f }{ \|f\|_p^p } \geq \affineeigenvalue.
\]	

For the proof of (b), consider a minimizer $f$ of $\affineeigenvalue$. The inequality \eqref{ineq_reversezhang} leads us to

\[
\affineeigenvalue = \frac { \E f^p }{ \|f\|_p^p } \geq \frac {\constRZ^p \|f\|_p^{p \frac{n-1}n} \gnorm{f}^{p/n} }{ \|f\|_p^p } \geq \constRZ^p \classicaleigenvalue^{1/n}.
\]
		
Finally, we prove the assertion (c). Let $f$ be a minimizer of $\affineeigenvalue$. By Proposition \ref{thm_chino}, there exists a matrix $A_0 \in \sl$ such that $f_{A_0}(x) = f(A_0^{-1} x)$ satisfies $\E f \geq \constchino \gnorm{f_{A_0}}$. Then,

\[
\constchino \min_{A \in \sl} \classicaleigenvalue[A \domain] \leq \constchino \classicaleigenvalue[A_0 \domain] \leq \constchino \frac{\gnorm{f_{A_0}}^p}{\|f_{A_0}\|_p^p} \leq \frac { \E^p f }{ \|f\|_p^p } = \affineeigenvalue.
\]	
\end{proof}

\begin{proof}[Proof of Theorem \ref{thm_equalitycase}]

For the proof of (a), assume first that $\domain = \B$.
		Arguing with the spherically decreasing rearrangement (see Proposition \ref{thm_affinePZ} of the section 2), it follows that $\affineeigenvalue[\B]$ admits a radial eigenfunction $f_p \in \W[\B]$. Then, the radial symmetry yields
		
\[
\affineeigenvalue[\B] = \frac{\E^p f_p}{\|f_p\|^p_p} = \frac{\|\nabla f_p\|^p_p}{\|f_p\|^p_p} \geq \classicaleigenvalue[\B],
\]
so that $\affineeigenvalue[\B] = \classicaleigenvalue[\B]$.
		
Conversely, assume equality in (a) for some open subset $\domain \subset \R^n$.
		Let $f_p \in \W \cap C^{1,\alpha}(\domain)$ be a positive eigenfunction of $\Delta_p$ corresponding to $\classicaleigenvalue$. Using the variational characterization of eigenvalues via minimization and the assumption $\affineeigenvalue = \classicaleigenvalue$, one easily concludes that $f_p \in \W \subset W^{1,p}(\Rn)$ satisfies $\E f_p = \gnorm{f_p}$. But the latter implies that $f_p$ is radial on $\R^n$. Since $f_p$ is positive in $\domain$, it follows that $\domain$ is a ball centered at the origin.

For the proof of (b), notice that $\classicaleigenone$ is equal to the Cheeger constant of $\domain$
	\[h_1(\domain) = \inf_{C \subseteq \domain} \frac {S(C)}{\vol(C)}.\]
Take $f$ to be any minimizer of $\classicaleigenone$.
	From the classical theory of Cheeger sets, we know (see \cite[Theorem 8]{kawohl2003isoperimetric}) that $f$ can be taken to be the characteristic function of a so-called Cheeger set $K \subseteq \domain$ of finite perimeter.
	Since
	\[
		\classicaleigenone = \frac{\gnorm[1] f}{\|f\|_1} \geq \frac{\E[1]f}{\|f\|_1} \geq \affineeigenone,
	\]
        we have that $f$ must be a radial function and thus $K$ must be a ball.

	It is known (see \cite[Remark 7]{kawohl2003isoperimetric}) that the mean curvature of the surface $\partial K$ at the interior points of $\domain$ must equal $\frac 1{n-1} \classicaleigenone$.
	But then a simple computation shows
	\[\classicaleigenone = \frac{S(\partial K)}{\vol(K)} = \frac nr,\]
        where $r>0$ is the radius of $K$, whereas the mean curvature of a sphere of radius $r$ is $\frac 1r$.
	This implies that there are no points of $\partial K$ that are interior to $\domain$, and so $\partial K \subseteq \partial \domain$.
\end{proof}

\begin{proof}[Proof of Theorem \ref{thm_affineCheegerSet}]
	For any $f \in BV(\domain)$, denote $K_t = \{x \in \R^n \suchthat f(x) \geq t\}$.
	We recall that
	\[\E[1]\chi_{K_t} = \left(c_{n,1} \vol(\Pi^\circ K_t)\right)^{-1/n}.\]
	By the co-area formula and Minkowski's integral inequality, we derive
	\begin{align}
		\E[1]f
		&\geq \left( c_{n,1} \int_{\S} \left( \int_0^\infty h_{\Pi K_t}(\xi) dt \right)^{-n} d\xi \right)^{-1/n}\\
		&\geq \int_0^\infty \left( c_{n,1} \int_{\S} h_{\Pi K_t}(\xi)^{-n} d\xi \right)^{-1/n} dt  \\
		&\geq \int_0^\infty \E[1]\chi_{K_t} dt.\\
	\end{align}
	
	Now if $f$ is a minimizer of $\affineeigenone$,
	\begin{align}
		0 &= \E[1]f - \affineeigenone \|f\|_1\\
		&\geq \int_0^\infty ( \E[1]\chi_{K_t} - \affineeigenone \vol(K_t)) dt\\
		&\geq 0.\\
	\end{align}
	Then, for almost every $t \in (0, \sup f)$, the function $\chi_{K_t}$ is a minimizer of $\affineeigenone$.
	It is clear that any of these $K_t$ satisfies the statement of the theorem.

	For the last statement, given $A \in \gl$ it suffices to note that (with the notation of \eqref{def_affineeigenvalue}),
	\[ R^{{\mathcal A}}_1(\chi_{AK}) = \det(A)^{-\frac 1n} R^{{\mathcal A}}_1(\chi_K). \]
\end{proof}

\section{Open problems}
\label{sec_openproblems}

Below we present some issues closely related to our results that we consider to be of great relevance and that could further deepen the understanding of the topics addressed in this work.

Let $\domain \subseteq \Rn$ be a bounded open set, $n \geq 2$ and $p \geq 1$.

\begin{enumerate}
\item We highlight some first issues related to the spectral theory satisfied by the operator $\affplap$ for $p > 1$.

\begin{itemize}
		\item[$\bullet$] Is it possible to characterize all eigenfunctions corresponding to $\affineeigenvalue$? Is the first affine eigenvalue simple?

        \item[$\bullet$] Is $\affineeigenvalue$ the unique eigenvalue which admits positive eigenfunction?

        \item[$\bullet$] Does $\affplap$ have any spectral gap?
\end{itemize}

These questions are by far not trivial even for balls where the first affine and classical eigenvalues coincide.

\item By Theorem \ref{thm_reversezhang}, there exists a best constant $A_{n,p}(\domain)$ such that, for any $f :\Rn \to \R$ a $C^1$ function with support in $\domain$,
	
\[
\E f \geq A_{n,p}(\domain) \|f\|_p^{\frac {n-1}{n}} \gnorm{f }^{1/n}.
\]
What is the value of $A_{n,p}(\domain)$? Is there any extremal function associated to $A_{n,p}(\domain)$? If so, is it smooth for $p > 1$?

\item The above inequality readily implies that

\[
\affineeigenvalue \geq A_{n,p}(\domain)^p \classicaleigenvalue^{1/n}.
\]
Let then $B_{n,p}(\domain)$ be the best constant associated to this equality, this is

\[
\affineeigenvalue \geq B_{n,p}(\domain) \classicaleigenvalue^{1/n}.
\]
Clearly, $B_{n,p}(\domain) \geq A_{n,p}(\domain)^p$. Does it occur equality at least for balls? Is the ball $\B$ an extremal domain for the above inequality? If so, the value of $B_{n,p}(\B)$ would be $\classicaleigenvalue[\B]^{(n-1)/n}$. What are all extremal domains for the above inequality?

	\item
		It is known that for the case $p = 1$, the classical eigenfunctions are characteristic functions of the so-called Cheeger sets of $\domain$.
		Theorem \ref{thm_affineCheegerSet} suggests that the theory of Cheeger sets could be developed in the affine case.
		For example: If $f$ is smooth, one can interpret $\Delta^1 f (x)$ as the mean curvature of the level set of $f$, and this leads to the characterization of the boundary structure of the Cheeger sets near the regular points.
		
	\begin{itemize}
		\item[$\bullet$] Is there a similar geometrical interpretation of $\affplap[1] f$?
		\item[$\bullet$] Are the affine Cheeger sets convex when $\domain$ is convex?
		\item[$\bullet$] If so, what other properties hold for the ``affine Cheeger sets'' in $\domain$?
	\end{itemize}

	\item
		We emphasize that, unlike the classical Cheeger sets, the affine Cheeger sets are in position of maximal volume.
		The following questions arise naturally:
		\begin{itemize}
			\item[$\bullet$] Is there a modification of the variational problem \eqref{def_affineCheegerConstant} whose minimizer sets are ellipsoids?
			\item[$\bullet$] Is there a characterization of spectral type (in term of eigenvalues of some operators) of John's position?
		\end{itemize}

	\item
		We ask whether inequality \eqref{ineq_reversezhang} can be improved to
		\[\E f \geq \constRZ \|f\|_q^{\frac {n-1}{n}} \gnorm{f }^{1/n}\]
		for some parameter $q > p$.

		A positive answer would allow to prove existence of minimizers and compactness to mixed variational problems.
		Some work has been done in \cite{schindler2018compactness} in this direction for the case $p = 2$ and $2 < q < n$, suggesting that the previous inequality could be valid.
	
	\item As mentioned in the introduction, one can see that $\affineeigenvalue$ is bounded from above when $\Omega$ ranges over all convex sets of a fixed volume.
		We ask whether the maximizers of $\affineeigenvalue$ can be identified.
\end{enumerate}

\noindent {\bf Acknowledgments:}
The first author was partially supported by CNPq (PQ 301203/2017-2) and Fapemig (APQ-01454-15). The second author was partially supported by CNPq (PQ 307471/2019-5 e Universal 428076/2018-1), FAPERJ (JCNE 236508). The third author was partially supported by CNPq (PQ 302670/2019-0, Universal 429870/2018-3) and Fapemig (PPM-00561-18).

\bibliographystyle{plain}
\bibliography{ref}

\end{document}